\documentclass[12pt,letterpaper]{amsart}

\usepackage[utf8]{inputenc}
\usepackage{enumerate}%
\usepackage{amsmath,amssymb}
\usepackage{cite}
\usepackage{comment}
\usepackage{amsfonts}
\usepackage{mathrsfs}
\usepackage{xcolor}
\usepackage[colorlinks, linkcolor = blue, anchorcolor = blue, citecolor = blue]{hyperref}
\usepackage{hyperref}
\usepackage{cleveref} 

\theoremstyle{definition}
\newtheorem{definition}{Definition}[section]
\newtheorem{remark}[definition]{Remark}

\theoremstyle{plain}
\newtheorem{theorem}[definition]{Theorem}
\newtheorem*{conjecture}{Conjecture}

\newtheorem{lemma}[definition]{Lemma}
\newtheorem{proposition}[definition]{Proposition}
\newtheorem{corollary}[definition]{Corollary}

\numberwithin{equation}{section}

\renewcommand{\S}{\mathbb{S}}
\newcommand{\Sph}{\mathbb{S}}

\newcommand{\vol}{V}
\newcommand{\dd}{\,\mathrm{d}}
\newcommand{\R}{\mathbb R}
\newcommand{\inter}{\operatorname{int}}
\newcommand{\supp}{\operatorname{supp}}
\newcommand{\diag}{\operatorname{diag}}
\newcommand{\Id}{\operatorname{Id}}
\title[Logarithmic Brunn--Minkowski under reflections]{Logarithmic Brunn--Minkowski inequality\texorpdfstring{\\}{}under $n-2$ reflection symmetries}
\author{XiaoRui Lu}
\address{School of Mathematical Sciences, Zhejiang University, Hangzhou 310058, China}
\email{12435035@zju.edu.cn}
\subjclass[2020]{Primary 52A40; Secondary 52A20.}
\keywords{Logarithmic Brunn--Minkowski inequality, reflection symmetry, equality case, cone volume measure, body of revolution}
\date{}
\begin{document}
\begin{abstract}

      For origin-symmetric convex bodies in $\R^n$ invariant under $n-2$ orthogonal hyperplanes
reflections, we establish the logarithmic Brunn--Minkowski inequality and classify all equality
cases without regularity assumptions. We derive the logarithmic Minkowski inequality and rigidity
for cone-volume measures, including an explicit classification for bodies of revolution and
uniqueness for simplicial polytopes in the prescribed symmetry class.
\end{abstract}
\maketitle

\section{Introduction}

The classical Brunn-Minkowski inequality is one of the fundamental principles of
convex geometry.  If $K,L\subset\R^n$ are convex bodies and
$0\leq\lambda\leq1$, then
\begin{equation}\label{eq:classical-bm}
 \vol\bigl((1-\lambda)K+\lambda L\bigr)^{\frac{1}{n}}
 \geq (1-\lambda)\vol(K)^{\frac{1}{n}}+\lambda\vol(L)^{\frac{1}{n}},
\end{equation}
with equality if and only if $K$ and $L$ are homothetic for $0 < \lambda < 1$.
It has become a starting point for several extensions of the classical Minkowski theory; see
\cite{Schneider2014} for background.

In the 1960s, Firey \cite{Firey1962} generalized the Minkowski combination of convex bodies to the $L_p$-Minkowski combination $(1 - \lambda) \cdot K +_p \lambda \cdot L$. When $p \geq 1$, for convex bodies $K$ and $L$ containing the origin in their interiors, it is defined as the convex body with support function
 $h_{(1-\lambda)\cdot K+_p\lambda\cdot L}
   :=((1-\lambda)h_K^p+\lambda h_L^p)^{\frac{1}{p}}$. The associated Brunn--Minkowski--Firey inequality reads as follows for $p \geq 1$,
\begin{equation}\label{eq:firey-bmi}
 \vol\bigl((1-\lambda)\cdot K+_p\lambda\cdot L\bigr)^{\frac{p}{n}}
 \geq (1-\lambda)\vol(K)^{\frac{p}{n}}+\lambda\vol(L)^{\frac{p}{n}},
\end{equation}
which has an equivalent form
\begin{equation}\label{eq:firey-bmieq}
 \vol\bigl((1-\lambda)\cdot K+_p\lambda\cdot L\bigr)
 \geq \vol(K)^{1-\lambda}\vol(L)^{\lambda},
\end{equation}
Lutwak's work in the 1990s placed this inequality, the $L_p$ mixed volumes,
and the $L_p$-Minkowski problem in a systematic theory
\cite{Lutwak1993,Lutwak1996}.

\indent However, when $p \in (0, 1)$, $((1-\lambda)h_K^p+\lambda h_L^p)^{\frac{1}{p}}$ may not be a support function for any convex body. Böröczky, Lutwak, Yang and Zhang \cite{BLYZ2012} found a natural generalization of $L_p$-Minkowski sum as follows
\[
(1 - \lambda) \cdot K +_p \lambda \cdot L := \bigcap_{x \in \mathbb{S}^{n-1}} \{z \in \mathbb{R}^n : x \cdot z \le ((1 - \lambda)h_K^p(x) + \lambda h_L^p(x))^{\frac{1}{p}}\}.
\]
It is natural to further define the log-Minkowski combination ($L_0$-Minkowski sum) $(1 - \lambda) \cdot K +_0 \lambda \cdot L$, by
\[
(1 - \lambda) \cdot K +_0 \lambda \cdot L := \bigcap_{x \in \mathbb{S}^{n-1}} \{z \in \mathbb{R}^n : x \cdot z \le h_K(x)^{1-\lambda}h_L(x)^\lambda\}.
\]

B\"or\"oczky, Lutwak, Yang and Zhang \cite{BLYZ2012} conjectured that, for
origin-symmetric convex bodies, \eqref{eq:firey-bmieq} remains true for
$0\leq p<1$. When $p=0$, \eqref{eq:firey-bmieq} is called log-Brunn-Minkowski inequality, stronger than the classical Brunn-Minkowski inequality and remains open in
dimensions $n\geq 3$.

\begin{conjecture} [log-Brunn-Minkowski conjecture\cite{BLYZ2012, BoroczkyKalantzopoulos2022}]
If $K$ and $L$ are origin-symmetric convex bodies in $\mathbb{R}^n$, then for any $\lambda \in (0,1)$, we have
\begin{equation}
V((1-\lambda)\cdot K +_0 \lambda \cdot L) \geq V(K)^{1-\lambda}V(L)^\lambda.
\end{equation}
In addition, equality holds if and only if $K = K_1 + \dots + K_m$ and $L = L_1 + \dots + L_m$ for compact convex sets $K_1, \dots, K_m, L_1, \dots, L_m$ of dimension at least one where $\sum_{i=1}^m \dim K_i = n$ and $K_i$ and $L_i$ are homothetic, $i = 1, \dots, m$.
\end{conjecture}

Let us briefly recall the known results and approaches to log-Brunn-Minkowski inequality.  The planar case was proved in
\cite{BLYZ2012}.
\begin{theorem}[Böröczky-Lutwak-Yang-Zhang]
\label{thm:BLYZ}
If $K$ and $L$ are origin-symmetric convex bodies in the plane, then for all $\lambda \in [0, 1]$,
\begin{equation}
V_2((1 - \lambda)\cdot K +_0 \lambda \cdot L) \ge V_2(K)^{1-\lambda}V_2(L)^{\lambda}. 
\end{equation}
When $\lambda \in (0, 1)$, equality holds if and only if $K$ and $L$ are dilates or $K$ and $L$ are parallelograms with parallel sides.
\end{theorem}
For unconditional bodies, the logarithmic inequality follows from the
coordinatewise geometric-mean inequality of Bollob\'as and Leader
\cite{BollobasLeader1995} and Cordero-Erausquin, Fradelizi, and Maurey \cite{CFM2004}; the equality analysis of
Saroglou \cite{Saroglou2015} was completed in \cite{BoroczkyKalantzopoulos2022}.
B\"or\"oczky and Kalantzopoulos \cite{BoroczkyKalantzopoulos2022} proved it for bodies invariant
under $n$ linear reflections whose fixed hyperplanes intersect only at
the origin. Rotem \cite{Rotem2014} established the complex case.
Yang and Zhang \cite{YangZhang2019} considered convex bodies in the relative Bonnesen class in $\R^3$.
 Tao, Xiong, and Xiong~\cite{TaoXiongXiong2023} focused on origin-symmetric cylinder cases in $\R^3$.
 Xi and Leng \cite{XiLeng2016} obtained a nonsymmetric planar extension
for bodies put in dilation position after suitable translations.

Kolesnikov and Milman \cite{KolesnikovMilman2022} developed a new method to derive local Brunn-Minkowski inequality by studying its infinitesimal form.
They obtained local $L_p$-Brunn--Minkowski inequalities for $C^2_+$ origin-symmetric convex bodies when
$p\in[1-c n^{-3/2},1)$. Chen, Huang, Li, and Liu \cite{ChenHuangLiLiu2020} and Putterman \cite{Putterman2021}
gave a direct equivalence between the local and global formulations based on PDE methods and strongly isomorphic polytopes, respectively.
Van Handel \cite{vanHandel2023} proved the local logarithmic Brunn--Minkowski inequality for zonoids.
Iffland \cite{iffland2026local} proved the local logarithmic Brunn--Minkowski inequality for bodies of revolution by operator theory. More recently, Milman \cite{Milman2025} interpreted the conjecture
spectrally in centro-affine differential geometry and obtained global results
under curvature pinching assumptions. Recently, Hu and Ivaki \cite{hu2026} developed a polyhedral version of Milman's centro-affine spectral
framework for full-dimensional origin-symmetric polytopes and optimized spectral estimate.



 We consider a different symmetry class: $n-2$
orthogonal hyperplane reflections, together with origin symmetry.
The remaining two-dimensional fibers retain origin-symmetry, so that
Theorem~\ref{thm:BLYZ} can be combined with the Pr\'ekopa--Leindler
inequality. We prove the following results without regularity assumption.

\begin{theorem}\label{thm:main}
Let $n\ge3$, and let $K,L\subset\R^n$ be origin-symmetric convex bodies
invariant under each coordinate reflection
\[
 \rho_i(x_1,\ldots,x_i,\ldots,x_n)
 =(x_1,\ldots,-x_i,\ldots,x_n),\qquad 1\le i\le n-2.
\]
Then, for every $0\leq\lambda\leq1$,
\begin{equation}\label{eq:main}
 V((1-\lambda)\cdot K+_0 \lambda \cdot L) \geq V(K)^{1-\lambda}V(L)^\lambda.
\end{equation}
Moreover, for any $0<\lambda<1$, equality holds if and only if there exist
nonzero subspaces $F_1,\ldots,F_m$, origin-symmetric convex
bodies $K_j\subset F_j$ with nonempty relative interiors, and $c_j>0$ such
that
\begin{equation}\label{eq:split}
 \R^n=F_1\oplus\cdots\oplus F_m,\qquad
 K=K_1+\cdots+K_m,\qquad L=c_1K_1+\cdots+c_mK_m.
\end{equation}
\end{theorem}
\begin{remark}
      The subspaces can be chosen invariant under every $\rho_i$.
They need not be mutually orthogonal. For every decomposition as in \eqref{eq:split},
\begin{equation}\label{eq:exactsum}
 (1-\lambda)\cdot K+_0 \lambda \cdot L=c_1^\lambda K_1+\cdots+c_m^\lambda K_m\qquad(0\le \lambda\le1).
\end{equation}
\end{remark}

An orthogonal change of coordinates gives the intrinsic formulation.

\begin{theorem}\label{thm:intrinsic}
Let $n\ge3$, let $v_1,\ldots,v_{n-2}\in\Sph^{n-1}$ be pairwise orthogonal, and set
$\rho_i x=x-2\langle x,v_i\rangle v_i$. If $K,L\subset\R^n$ are
origin-symmetric convex bodies with $\rho_iK=K$ and $\rho_iL=L$ for every
$i$, then \eqref{eq:main} holds. For $0<\lambda<1$, equality holds
exactly for the pairs in \eqref{eq:split}.
\end{theorem}

In particular, if either body is strictly convex, equality holds only for
dilates, since a direct sum of two bodies has a nontrivial
segment in its boundary.

\begin{corollary}\label{cor:revolution}
Let $n\ge3$, and let $K,L\subset\R^n$ be origin-symmetric convex bodies of
revolution about axes $\R a$ and $\R b$, respectively. Then
\eqref{eq:main} holds, with equality for $0<\lambda<1$ if and only if
\eqref{eq:split} holds. If either body is strictly convex, equality holds
if and only if $K,L$ are dilates.
\end{corollary}

Here a body of revolution is invariant under all rotations around a fixed axis.
 The axes in Corollary~\ref{cor:revolution}
may be neither parallel nor orthogonal, so they fail to match with the
reflection system required in \cite{BoroczkyKalantzopoulos2022}.

\begin{equation}\label{eq:conemeasure}
 dV_K=\frac1n h_K\dd S_K,\qquad V_K(\Sph^{n-1})=V(K).
\end{equation}
Thus $V_K$ is a measure, whereas $V(K)$ denotes volume of convex body $K$.

\begin{theorem}\label{thm:logmink}
Under the assumptions of Theorem~\ref{thm:intrinsic},
\begin{equation}\label{eq:logmink}
 \int_{\Sph^{n-1}}\log\frac{h_L}{h_K}\dd V_K
 \ge \frac{V(K)}n\log\frac{V(L)}{V(K)}.
\end{equation}
Equality holds if and only if \eqref{eq:split} holds.
\end{theorem}

\begin{corollary}\label{cor:cone}
Under the assumptions of Theorem~\ref{thm:intrinsic}, $V_K=V_L$ if and
only if \eqref{eq:split} holds with
\begin{equation}\label{eq:productone}
 \prod_{j=1}^m c_j^{\dim F_j}=1.
\end{equation}
Consequently, $V_K=V_L$ implies $K=L$ if either body is strictly convex,
or if the measure has full support on $\Sph^{n-1}$.
\end{corollary}

The last assertion gives uniqueness within the prescribed reflection
class, without regularity assumption, whenever the prescribed cone
volume measure has a strictly positive continuous density. The direct-sum characterization is essential, since boxes with
the same coordinate directions and volume have the same cone volume
measure even when their side lengths differ.
The direct-sum decomposition also makes the revolution case explicit.

\begin{theorem}\label{thm:revolution-rigidity}
Let $n\ge3$, and $K,L\subset\mathbb R^n$ be origin-symmetric convex
bodies of revolution, with possibly different axes. For $0<\lambda<1$,
equality in the logarithmic Brunn--Minkowski inequality holds if and only if
$K,L$ are dilates, or they are coaxial circular cylinders. More explicitly,
 for some $a\in\mathbb S^{n-1}$ and
$h,r,\widetilde h,\widetilde r>0$,
\[
 K=[-ha,ha]+rB_{a^\perp},\qquad
 L=[-\widetilde h a,\widetilde h a]+\widetilde rB_{a^\perp},
\]
where $B_{a^\perp}$ is the Euclidean unit ball of $a^\perp$.
Moreover, $V_K=V_L$ if and only if $K=L$, or they are such cylinders with
$hr^{n-1}=\widetilde h\widetilde r^{n-1}$.
\end{theorem}

Thus every noncylindrical body of revolution is determined by its cone
volume measure within the class of origin-symmetric bodies of revolution.
Also every simplicial polytope with prescribed symmetries is determined by its cone
volume measure.

\begin{theorem}\label{thm:simplex}
Let $n\ge3$, and $K,L\subset\mathbb R^n$ be origin-symmetric convex bodies
satisfying
\(\rho_iK=K, \rho_iL=L, 1\le i\le n-2\).
Suppose that $K$ is a simplicial polytope, then for $0<\lambda<1$,
equality in the logarithmic Brunn--Minkowski inequality holds
if and only if $K, L$ are dilates.
Moreover, $V_K=V_L$ implies $K=L$.
\end{theorem}

Concerning the organization of the paper, \Cref{pre} records basic facts used below. \Cref{bmi}
proves the Brunn--Minkowski inequality. \Cref{equality} treats nonsmooth equality, and
\Cref{cors} proves the consequences concerning cone volume measure.

\section{Preliminaries}
\label{pre}
In this section, we collect some known facts about convex bodies. Good general references for the theory of convex bodies are provided by the books of Gardner \cite{Gardner2006GeometricTomography}, Gruber \cite{Gruber2007ConvexDiscreteGeometry}, Leichtweiss \cite{Leichtweiss1998AffineGeometry}, Schneider \cite{Schneider2014}, and Thompson \cite{Thompson1996MinkowskiGeometry}.

In the Euclidean space $\mathbb{R}^n$, we denote the standard inner product by $\langle \cdot, \cdot \rangle$, the Euclidean norm by $|\cdot|$, the $n$-dimensional volume by $V(\cdot)$, the $2$-dimensional volume/area by $V_2(\cdot)$, and the $k$-dimensional Hausdorff measure by $\mathcal{H}^k$. A convex body is a compact convex subset with nonempty interior. In this paper, unless otherwise stated, all convex bodies are assumed to contain the origin in their interior. We write $M|E$ to denote the orthogonal projection of a compact convex set $M$ onto a linear subspace $E$ in $\mathbb{R}^n$.

The \textit{support function} $h_K : \mathbb{R}^n \to \mathbb{R}$ of a compact convex body $K$ in $\mathbb{R}^n$ is defined, for $x \in \mathbb{R}^n$, by
\[
h_K(x) = \max \{ \langle x, y \rangle : y \in K \}.
\]
Note that support functions are positively homogeneous of degree one and subadditive.
Thus, $K$ can be defined as 
\begin{equation}\label{eq:supp}
     K=
\bigcap_{u\in\operatorname{supp}S_K}
\left\{
x\in\mathbb R^n:
\langle x,u\rangle\le h_K(u)
\right\}.
\end{equation}
 A vector $u \in \mathbb{R}^n \setminus \{o\}$ is an outer normal vector at a boundary point $x \in \partial K$ if
\(
\langle x, u \rangle = h_K(u),
\)
and it is called a unit outer normal if $u \in \S^{n-1}$. A boundary point is said to be \textit{regular} if it has only one unit normal vector, and is said to be \textit{singular} if it has more than one unit normal vector. It is well known that the set of all singular boundary points of a convex body has $\mathcal{H}^{n-1}$-measure equal to 0.

Let $K$ be a convex body in $\mathbb{R}^n$ and $\nu_K : \partial K \to S^{n-1}$ denote the generalized Gauss map. Given $p \in \mathbb{R}$, $S_p(K, \cdot)$, the $L_p$ \textit{surface area measure} of $K$, is a Borel measure on the unit sphere $\mathbb{S}^{n-1}$, given by
\[
    S_p(K,\omega) = \int\limits_{y\in\nu_K^{-1}(\omega)} (y \cdot \nu_K)^{1-p} d\mathcal{H}^{n-1}(y), \quad \text{for Borel set } \omega \subseteq \mathbb{S}^{n-1}.
\]
Two important cases are $S_1(K, \cdot)$ and $\frac{1}{n}S_0(K, \cdot)$, the former is called \textit{surface area measure} of $K$ (also denoted by $S_K$) and the latter is known as the \textit{cone volume measure} of $K$ (also denoted by $V_K$).

We recall that for $\lambda \in (0, 1)$, the $L_0$-sum of two convex bodies $K$ and $L$ in $\mathbb{R}^n$ is the Wulff shape
\begin{align*}
M_\lambda(K,L) := (1 - \lambda) \cdot K +_0 \lambda \cdot L &= \left\{ x \in \mathbb{R}^n : \langle x, u \rangle \le h_K(u)^{1-\lambda}h_L(u)^\lambda \ \forall u \in \S^{n-1} \right\} \\
&= \left\{ x \in \mathbb{R}^n : \langle x, u \rangle \le h_K(u)^{1-\lambda}h_L(u)^\lambda \ \forall u \in \mathbb{R}^n \right\}.
\end{align*}
It is well-known that \cite{BoroczkyKalantzopoulos2022} for any $u \in \mathbb{R}^n \setminus \{o\}$, if  $u$ is the outer normal at regular point $ z \in \partial ((1 - \lambda) \cdot K +_0 \lambda \cdot L)$,  then 
\begin{align}\label{eq:contact}
      h_{(1-\lambda)\cdot K +_0 \lambda \cdot L}(u) = h_K(u)^{1-\lambda}h_L(u)^\lambda. 
\end{align}

We note that the logarithmic sum is linear covariant. If $\Phi \in GL(n, \mathbb{R})$, then
\begin{equation} \label{eq:10}
\Phi[(1 - \lambda) \cdot K +_0 \lambda \cdot L] = (1 - \lambda) \cdot \Phi(K) +_0 \lambda \cdot \Phi(L).
\end{equation}
Therefore, if $K$ and $L$ are two convex bodies in $\mathbb{R}^n$ invariant under some subgroup $G \subset GL(n)$, then $(1-\lambda)\cdot K +_0 \lambda \cdot L$ is also invariant under $G$.

Suppose that the function $k_t(u) = k(t, u) : I \times S^{n-1} \to (0, \infty)$ is continuous, where $I \subset \mathbb{R}$
is an interval. For fixed $t \in I$, let
\[
K_t = \bigcap_{u \in S^{n-1}} \{x \in \mathbb{R}^n : x \cdot u \le k(t, u)\}
\]
be the Wulff shape (or Aleksandrov body) associated with the function $k_t$. It is well-known that
\begin{equation}
h_{K_t} \le k_t \quad \text{and} \quad h_{K_t} = k_t, \quad \text{a.e. w.r.t. } S_{K_t}, 
\end{equation}
for each $t \in I$. If $k_t$ is the support function of a convex body, then $h_{K_t} = k_t$, everywhere.
The following variant of Aleksandrov's Lemma \cite{Alexandrov1996SelectedWorks, HaberlLutwakYangZhang2010EvenOrlicz, Schneider2014} will be needed.

\begin{lemma}\label{lem:variation}
    Suppose $k(t, u) : I \times S^{n-1} \to (0, \infty)$ is continuous, where $I \subset \mathbb{R}$ is an open
interval. Suppose also that the convergence in
\[
\frac{\partial k(t, u)}{\partial t} = \lim_{s \to 0} \frac{k(t + s, u) - k(t, u)}{s}
\]
is uniform on $S^{n-1}$. If $\{K_t\}_{t \in I}$ is the family of Wulff shapes associated with $k_t$, then
\[
\frac{d V(K_t)}{d t} = \int_{S^{n-1}} \frac{\partial k(t, u)}{\partial t} dS_{K_t}(u).
\]
\end{lemma}

The following n-dimensional Pr\'ekopa-Leindler inequality is needed. 
\begin{theorem}[Pr\'ekopa-Leindler \cite{Prekopa1971,Prekopa1973LogarithmicConcaveFunctions}, Dubuc \cite{Dubuc1977}]
\label{thm:PL}
Let $0<\lambda<1$. If $f,g,h:\R^n\to[0,\infty)$ be measurable and
integrable functions satisfying $h((1-\lambda)s + \lambda t) \ge f(s)^{1-\lambda}g(t)^\lambda$ for any $ s,t \in \mathbb{R}^n$, then
\[
 \int_{\R^n}h
 \geq
 \left(\int_{\R^n}f\right)^{1-\lambda}
 \left(\int_{\R^n}g\right)^\lambda.
\]
When $ 0<\int_{\R^n}f, \int_{\R^n}g<+\infty$, if equality holds, then there exist $\tau \in \mathbb{R}^n$ and $a > 0$ such that for almost every $s \in \mathbb{R}^n$,
\begin{align}\label{eq:dubuc}
     g(s + \tau) = af(s), \qquad h(s + \lambda\tau) = f(s)^{1-\lambda}g(s + \tau)^\lambda.
\end{align}

\end{theorem}

Set $k=n-2$ and write
\[
 \R^n=E\oplus P,\qquad E=\operatorname{span}\{e_1,\ldots,e_k\},
 \qquad P=E^\perp,\qquad E_+=(0,\infty)^k.
\]
For $Q\subset E\oplus P$, its planar fiber is
$Q_x=\{y\in P:(x,y)\in Q\}$. Empty fibers have area zero.

\begin{lemma}\label{lem:regularity}
If $Q\subset E\oplus P$ is a convex body, then $Q_x$ has nonempty
interior in $P$ for every $x\in\inter(Q|E)$. On this domain,
$x\mapsto Q_x$ is locally Lipschitz in Hausdorff distance, and
$x\mapsto V_2(Q_x)$ is continuous.
\end{lemma}
\begin{proof}
For any $x\in\inter(Q|E)$, choose an interior point of $Q$ with projection $x_0 \ne x$.
 There exists some $x_1\in \inter(Q|E)$ and $0<t<1$ such that $x=(1-t)x_0+tx_1$.
 Assume that $z_0=(x_0,y_0)$ is an interior point of $Q$, $Q$ contains $B_{\R^n}(z_0,\epsilon)$ for some small $\epsilon$, thus $Q_{x_0}$ contains $B_P(y_0,\epsilon)$.
 Take $(x_1,y_1) \in Q$,  $t(x_1,y_1)+ (1-t) B_{\R^n}(z_0,\epsilon) \subset Q$ due to the convexity, leading to $B_P\bigl((1-t)y_0+ty_1,\,(1-t)\varepsilon\bigr)\subset Q_x.$
 Convexity also gives $(1-t)Q_x+tQ_{x'}\subset Q_{(1-t)x+tx'}$,
so $x\mapsto h_{Q_x}(v)$ is concave for each $v\in P$.
Set $|v|=1$, take $B(X_0,2r) \subset \operatorname{int}(Q|E)$ and sufficiently large $R$, $|h_{Q_x}(v)| \le R$ for every $x \in B(x_0,2r)$.
Since $h_{Q_x}(v)$ is concave, $|h_{Q_x}(v)-h_{Q_{x}}(v)| \le C(R,r)|x-x'|$ for $x,x' \in B(x_0,r)$.
Now use $d_H(Q_x,Q_{x'})=\sup_{|v|=1}|h_{Q_x}(v)-h_{Q_{x'}}(v)|$
and continuity of area under Hausdorff convergence, $x\mapsto V_2(Q_x)$ is continuous.
\end{proof}

\section{Proof of logarithmic Brunn--Minkowski inequality}
\label{bmi}
Throughout this section the hypotheses of Theorem~\ref{thm:main} hold.
Origin symmetry and the $k$ coordinate reflections imply invariance under
$(x,y)\mapsto(x,-y)$. Thus every nonempty fiber is origin-symmetric.
 For $x,x'\in E_+$, denote
$x\cdot_t x'=(x_1^{1-t}(x'_1)^t,\ldots,x_k^{1-t}(x'_k)^t)$.

\begin{lemma}\label{lem:inclusion}
Let $x,x'\in E_+$. If $V_2(K_x)V_2(L_{x'})>0$ and $0<t<1$, then
\begin{equation}\label{eq:fiberinclusion}
 M_t(K_x,L_{x'})\subset M_t(K,L)_{x\cdot_t x'}.
\end{equation}
\end{lemma}
\begin{proof}
Take $w\in M_t(K_x,L_{x'})$ and $(u,v)\in E\oplus P$. Set
$\bar u=(|u_1|,\ldots,|u_k|)$. H\"older's inequality gives
\begin{align*}
 \langle x\cdot_t x',u\rangle+\langle w,v\rangle
 &\le \sum_{i=1}^k(\bar u_ix_i)^{1-t}(\bar u_ix'_i)^t
       +h_{K_x}(v)^{1-t}h_{L_{x'}}(v)^t\\
 &\le(\langle x,\bar u\rangle+h_{K_x}(v))^{1-t}
       (\langle x',\bar u\rangle+h_{L_{x'}}(v))^t\\
 &\le h_K(\bar u,v)^{1-t}h_L(\bar u,v)^t
  =h_K(u,v)^{1-t}h_L(u,v)^t.
\end{align*}
The last equality uses the coordinate reflections. This is precisely
the inequality defining the fiber on the right of \eqref{eq:fiberinclusion}.
\end{proof}

\begin{proof}[Proof of the inequality in Theorem~\ref{thm:main}]
Fix $0<\lambda<1$ and set $M=M_\lambda(K,L)$. For $s\in\R^k$, write
$e^s=(e^{s_1},\ldots,e^{s_k})$ and $\langle s, \mathbf{1} \rangle = \Sigma_{i=1}^{k} s_i$. Define
\begin{equation}\label{eq:densities}
 G(s)=e^{\langle s, \mathbf{1} \rangle}V_2(K_{e^s}),\quad
 H(s)=e^{\langle s, \mathbf{1} \rangle}V_2(L_{e^s}),\quad
 F(s)=e^{\langle s, \mathbf{1} \rangle}V_2(M_{e^s}).
\end{equation}
By Lemma~\ref{lem:inclusion} and Theorem~\ref{thm:BLYZ},
\[
 F((1-\lambda)s+\lambda r)\ge G(s)^{1-\lambda}H(r)^\lambda.
\]
If one fiber has area zero this is immediate. Applying \Cref{thm:PL} to the functions $F, G$, and $H$, and changing variables back to $E_+$ gives:
\begin{align*}
    V(M \cap (E_+ \times P)) &= \int_{\mathbb{R}^{n-2}} F(r) \, dr \\
    &\ge \left( \int_{\mathbb{R}^{n-2}} G(s) \, ds \right)^{1-\lambda} \left( \int_{\mathbb{R}^{n-2}} H(t) \, dt \right)^\lambda \\
    &= V(K \cap (E_+ \times P))^{1-\lambda} V(L \cap (E_+ \times P))^\lambda.
\end{align*}

The Wulff body $M$ inherits reflection symmetries from $K,L$.
Coordinate hyperplanes have zero volume, so the three orthant volumes above
are equal to $2^{-k}V(M)$, $2^{-k}V(K)$, and $2^{-k}V(L)$, respectively.  The common
coefficient $2^{-k}$ cancels and yields \eqref{eq:main}. At
$\lambda=0,1$ the assertion is an identity.
\end{proof}

\section{Equality without boundary regularity}
\label{equality}
In this section, fix $0<\lambda<1$, suppose that equality holds in
\eqref{eq:main} and keep $M=M_\lambda(K,L)$.

\begin{lemma}\label{lem:data}
There are $a_i>0$, $A=\diag(a_1,\ldots,a_k)$, and $d>0$ such that
\begin{equation}\label{eq:projections}
 L|E=A(K|E),\qquad M|E=A^\lambda(K|E).
\end{equation}
For every $x\in D:=\inter(K|E)\cap E_+$,
\begin{align}
 &V_2(L_{Ax})=d^2V_2(K_x),&
 V_2(M_{A^\lambda x})&=d^{2\lambda}V_2(K_x),\label{eq:areas}\\
 &M_{A^\lambda x}=M_\lambda(K_x,L_{Ax}).\label{eq:fiberexact}
\end{align}
Each pair $(K_x,L_{Ax})$ consists either of dilates with $L_{Ax}=dK_x$,
or of parallelograms with parallel sides.
\end{lemma}
\begin{proof}
Equality in the preceding proof and \eqref{eq:dubuc} yield
$H(s+\tau)=CG(s)$ and $F(s+\lambda\tau)=C^\lambda G(s)$ almost
everywhere. Integration gives $C=V(L)/V(K)$. Set $a_i=e^{\tau_i}$ and
$d^2=C/\det A$, \eqref{eq:densities} gives
\eqref{eq:areas} almost everywhere. \eqref{eq:areas} implies $V_2(L_{Ax})>0$ is equivalent to $V_2(K_x)>0$, thus $\operatorname{int}(L|E)\cap E_+ = A\bigl(\operatorname{int}(K|E)\cap E_+\bigr)$ up to a null set.
Therefore, $L|E=A(K|E)$ by taking closure and using reflection symmetries. Similarly, $M|E=A^\lambda(K|E)$.
By \Cref{lem:regularity}, three fiber-area functions in \eqref{eq:areas} is continuous, thus \eqref{eq:areas} holds for every $x\in D$.
Applying \Cref{lem:inclusion} and \Cref{thm:BLYZ}, we have
\[
 V_2(M_{A^\lambda x})\ge V_2(M_\lambda(K_x,L_{Ax}))
 \ge V_2(K_x)^{1-\lambda}V_2(L_{Ax})^\lambda
 =d^{2\lambda}V_2(K_x).
\]
Comparing with (\ref{eq:areas}), all intermediate inequalities are equalities.
Thus, for every $x \in D$, we have \(M_{A^\lambda x} =M_\lambda(K_x,L_{Ax}).\)
Theorem~\ref{thm:BLYZ} gives the alternatives.
\end{proof}

We next show that the parallelogram cannot vary freely with
the base point.

\begin{lemma}\label{lem:directions}
Let $Q$ be convex body and $W\subset\inter(Q|E)$ be connected and open. If every $Q_x$, $x\in W$,
is a nondegenerate origin-symmetric parallelogram, its two
edge directions are independent of $x$.
\end{lemma}
\begin{proof}
By \Cref{lem:regularity}, the fiber $Q_x$ is locally Lipschitz continuously in Hausdorff distance.
 Since all nearby fibers are nondegenerate parallelograms, near any $x_0\in W$, the four vertices can be labeled as locally
Lipschitz functions of $x$, thus
\[
 Q_x=[-b(x),b(x)]+[-c(x),c(x)],
\]
with locally Lipschitz independent vectors $b(x),c(x)$.

At a point of differentiability, fix a base direction $\xi$ and a nonzero
$v\perp b(x)$. Then $\langle v,c(x)\rangle\ne0$, and
\[
 h_{Q_{x+t\xi}}(v)+h_{Q_{x-t\xi}}(v)-2h_{Q_x}(v)
 =2|t|\,|\langle v,D_\xi b(x)\rangle|+o(|t|).
\]
Concavity of $h_{Q_x}$ makes the left side nonpositive.
Thus $D_\xi b(x)$ is parallel to $b(x)$, and the derivative of $b/|b|$
vanishes almost everywhere. A locally Lipschitz function with zero weak
gradient is constant on each small ball \cite{brezis2011}. The same holds for $c/|c|$.
The edge directions of parallelogram is consequently locally constant,
and connectedness finishes the proof.
\end{proof}

\begin{proposition}\label{prop:map}
Either $L=\diag(A,dI_2)K$, or, after a linear transformation of coordinates
on $P$, both bodies are unconditional and
\begin{equation}\label{eq:map}
 L=\diag(a_1,\ldots,a_k,\alpha,\beta)K,
 \qquad \alpha\beta=d^2,\quad\alpha\ne\beta.
\end{equation}
\end{proposition}
\begin{proof}
The set $D$ of Lemma~\ref{lem:data} is open, convex, and nonempty. The set
$U=\{x\in D:L_{Ax}\ne dK_x\}$ is relatively open by fiber continuity.
If it is empty, the reflections and closure give the
first conclusion. Otherwise, let $W$ be a connected component of $U$.
By Lemma~\ref{lem:directions}, the parallelogram fibers over $W$ have
fixed directions, which transfer to coordinate
directions after a linear map on $P$. This map keeps the original reflections; by
\eqref{eq:10}, the equality hypothesis and all fiber area ratios in \Cref{lem:data}
are preserved. Write
\[
 K_x=\prod_{j=1}^2[-b_j(x),b_j(x)],\qquad
 L_{Ax}=\prod_{j=1}^2[-\widetilde b_j(x),\widetilde b_j(x)]
 \qquad(x\in W).
\]
The functions are positive and continuous. By \eqref{eq:fiberexact},
$M_{A^\lambda x}$ is the rectangle with half-side lengths
$b_j(x)^{1-\lambda}\widetilde b_j(x)^\lambda$.

Choose a bounded open box $B\subset\R^k$ with $I:=\exp B\subset W$, and
consider the open sets
\begin{align*}
 X&=\inter K\cap(I\times(0,\infty)^2),\\
 Y&=\inter L\cap(AI\times(0,\infty)^2),\\
 Z&=\inter M\cap(A^\lambda I\times(0,\infty)^2).
\end{align*}
Indeed, for $z\in X$, $w\in Y$, and $u\ne0$, 
H\"older's inequality gives
\begin{equation}\label{eq:positiveinclusion}
 \sum_{j=1}^n u_jz_j^{1-\lambda}w_j^\lambda
 \le\Big(\sum_j|u_j|z_j\Big)^{1-\lambda}
      \Big(\sum_j|u_j|w_j\Big)^\lambda
 <h_K(u)^{1-\lambda}h_L(u)^\lambda.
\end{equation}
Every coordinate sign change of $z,w$ still belongs to $\inter K$ and $\inter L$ respectively, resulting in the last inequality. Thus $X\cdot_\lambda Y=\{z\cdot_\lambda w | z\in X, w\in Y\} \subset Z$ and
\(
(1-\lambda)\log X+\lambda\log Y\subset\log Z
\).
$M_t(K,K)=K$ shows that $\log X$ is convex. Likewise $\log Y$ is
convex. 
By \eqref{eq:areas}, 
\begin{equation}\label{eq:localvol}
 V(Y)=d^2(\det A)V(X),\qquad
 V(Z)=d^{2\lambda}(\det A)^\lambda V(X).
\end{equation}
Apply \Cref{thm:PL} on $\R^n$ to
$f(s)=e^{\langle s, \mathbf{1} \rangle}\mathbf1_{\log X}(s)$,
$g(s)=e^{\langle s, \mathbf{1} \rangle}\mathbf1_{\log Y}(s)$, and
$h(s)=e^{\langle s, \mathbf{1} \rangle}\mathbf1_{\log Z}(s)$. 
\eqref{eq:positiveinclusion} gives its hypothesis and \eqref{eq:localvol}
gives equality. Dubuc's theorem yields a vector $\eta$ with
$\log Y=\log X+\eta$ up to a null set. $\log Y$ and $\log X$ are open and convex,
so equality holds exactly. Projecting onto the first $k$ coordinates
gives $B+\log a=B+\eta_E$. Boundedness of $B$ implies
$\eta_E=\log a$. Thus, with $\alpha=e^{\eta_{n-1}}$ and
$\beta=e^{\eta_n}$,
\[
 \widetilde b_1(x)=\alpha b_1(x),\qquad
 \widetilde b_2(x)=\beta b_2(x)\qquad(x\in I),\qquad
 \alpha\beta=d^2.
\]
The two ratios are continuous and locally constant on $W$, hence constant
on $W$. They are distinct because $W\subset U$. If $x_\ell\in W$
converges to $x\in D$, continuity gives nondegenerate limiting rectangles
with the same distinct ratios. Thus $x\in U$, and a small connected
neighborhood in $U$ shows $x\in W$. Therefore $W$ is also closed in
$D$, $W=D$. Reflections and closure extend the rectangular fiber
to the entire bodies, proving unconditionality and
\eqref{eq:map}.
\end{proof}

\begin{lemma}\label{lem:blocks}
Let $\R^n=H_1\oplus\cdots\oplus H_r$ be an orthogonal decomposition
and $\Pi_j$ be the orthogonal projection onto $H_j$. Assume that $K$ is invariant under each subspace reflection $R_j = Id - 2\Pi_j$.
If $T|_{H_j}=t_j\Id$ where $t_j>0$, and $S=T^\lambda$, then
\begin{equation}\label{eq:blockincl}
 SK\subset M_\lambda(K,TK).
\end{equation}
If $ SK = M_\lambda(K,TK)$, then
\begin{equation}\label{eq:blocksplitting}
 K=\bigoplus_{\rho\in\sigma(T)}(K|F_\rho),\qquad
 F_\rho=\ker(T-\rho\Id).
\end{equation}
\end{lemma}
\begin{proof}
Write $z\in K$, $u \in \R^n$, and
$b_j=|\langle z_j,u_j\rangle|$. Origin-symmetries in each subspace and H\"older's inequality give
\begin{equation}\label{eq:blockchain}
 \langle Sz,u\rangle\le\sum_jt_j^\lambda b_j
 \le\Big(\sum_jb_j\Big)^{1-\lambda}\Big(\sum_jt_jb_j\Big)^\lambda
 \le h_K(u)^{1-\lambda}h_K(Tu)^\lambda.
\end{equation}
Since $T$ is self-adjoint this proves \eqref{eq:blockincl}.

For almost every $z\in\partial K$, both $z$ and $Sz$ are
regular, and \eqref{eq:contact} holds at the normal
of $SK$ at $Sz$. Choose such a point, and let $u\ne0$ be an outer normal
there. Then $q=Su$ is an outer normal of $K$ at $z$, and
\begin{equation}\label{eq:blockcontact}
 \langle z,Su\rangle=\langle Sz,u\rangle=h_K(u)^{1-\lambda}h_K(Tu)^\lambda.
\end{equation}
Since $R_j$ reverse subspace $H_j$. \(\langle R_jz,q\rangle\le\langle z,q\rangle\) and $R_jz\in K$
give $\langle z_j,q_j\rangle\ge0$. If it vanished while $q_j\ne0$,
then both $q$ and $R_jq$ would be distinct equal-length outer normals at
$R_jz$. This contradicts regularity of $R_jz$. Hence
\[
 \langle z_j,q_j\rangle>0\quad\text{if and only if}\quad q_j\ne0.
\]
\(b_j
=|\langle z_j,u_j\rangle|
=t_j^{-\lambda}\langle z_j,q_j\rangle\),
thus $b_j$ is nonzero in \eqref{eq:blockchain} if and only if $q_j\ne0$. Equation~\eqref{eq:blockcontact} forces equality
in \eqref{eq:blockchain}, so all $t_j$ on these subspaces coincide.
Consequently, the outer normal $q$ of regular point $z\in\partial K$ belongs to one eigenspace of $T$
\[
 q=\sum_{j:\,t_j=\rho\in\sigma(T)}q_j\in F_\rho,
\]
and \[
 \supp S_K\subset \bigcup_{\rho\in\sigma(T)}
       \bigl(F_\rho\cap\Sph^{n-1}\bigr).
\]

We now show that this inclusion forces the asserted decomposition.
Set
\[
 \widehat K:=\bigoplus_{\rho\in\sigma(T)}(K|F_\rho).
\]
The eigenspaces of $T$ form an orthogonal decomposition of
$\mathbb R^n$, so every point of $K$ is the sum of its orthogonal
projections onto these eigenspaces. Thus $K\subset\widehat K$.
Conversely, let
\[
 y=\sum_{\rho\in\sigma(T)}y_\rho\in\widehat K,
 \qquad y_\rho\in K|F_\rho.
\]
For any $u\in\supp S_K$, there exists an eigenvalue $\rho_0$ such that $u\in F_{\rho_0}$.
Orthogonality then yields
\(
 \langle y,u\rangle
 =\langle y_{\rho_0},u\rangle
 \le h_{K|F_{\rho_0}}(u)
 =h_K(u).
\)
Thus $y\in K$, $\widehat K\subset K$, and consequently
\[
 K=\bigoplus_{\rho\in\sigma(T)}(K|F_\rho),
\]
proving \eqref{eq:blocksplitting}.
\end{proof}

\begin{proof}[Completion of Theorem~\ref{thm:main}]
Proposition~\ref{prop:map} gives $L=TK$ with two kinds of decomposition:
either $\R e_1,\ldots,\R e_k,P$, or, after a linear transformation on $P$,
all coordinate axes. Lemma~\ref{lem:blocks} therefore applies, and the
equality hypothesis gives
\[
 V(M_\lambda(K,L))=(\det T)^\lambda V(K)=V(T^\lambda K).
\]
It follows that $K$ splits over the distinct eigenspaces of $T$, while
$L$ scales each eigenspace by its corresponding eigenvalue. All original symmetries
commute with the eigenspace projections in the transformed coordinates.
Mapping back yields \eqref{eq:split}, with each subspace
and projection invariant under every prescribed reflection.

Conversely, suppose \eqref{eq:split} holds. By \eqref{eq:10}, we
may make the subspaces orthogonal. For $u=\sum_ju_j$,
$u_j\in F_j$, H\"older's inequality gives
\[
 \sum_jc_j^th_{K_j}(u_j)
 \le\Big(\sum_jh_{K_j}(u_j)\Big)^{1-t}
      \Big(\sum_jc_jh_{K_j}(u_j)\Big)^t.
\]
This proves $\sum_jc_j^tK_j\subset M_t(K,L)$.
For any \(x\in M_t(K,L)\), take \(u\in F_j\), \(\langle x_j,u\rangle
\le h_{K_j}(u)^{1-t}
\bigl(c_jh_{K_j}(u)\bigr)^t
=c_j^th_{K_j}(u),\)
hence $M_t(K,L)\subset \sum_jc_j^tK_j$, proving
\eqref{eq:exactsum}. If $d_j=\dim F_j$, the product formula yields
\[
 V(L)=V(K)\prod_jc_j^{d_j},\qquad
 V(M_t(K,L))=V(K)\prod_jc_j^{td_j}.
\]
Thus equality holds for every $t\in[0,1]$. Linear transformation completes
the argument in the original subspaces.
\end{proof}

\begin{proof}[Proofs of Theorem~\ref{thm:intrinsic} and
Corollary~\ref{cor:revolution}]
Choose $Q\in O(n)$ with $Qv_i=e_i$. Apply Theorem~\ref{thm:main} to
$QK,QL$ and use \eqref{eq:10}. This proves the intrinsic
statement and its equality classification. For the corollary, choose
$n-2$ orthonormal vectors in $(\operatorname{span}\{a,b\})^\perp$.
Their reflecting hyperplanes contain both axes and preserve on both
bodies. Theorem~\ref{thm:intrinsic} applies. Strict convexity rules out
decomposition in \eqref{eq:split}.
\end{proof}

\section{Logarithmic Minkowski inequality and cone volume rigidity}
\label{cors}
\begin{proof}[Proof of Theorem~\ref{thm:logmink}]
Set $K_t=M_t(K,L)$ and $\varphi(t)=\log V(K_t)$. These bodies retain
the prescribed symmetries, and their support functions satisfy
$h_{K_t}\le h_K^{1-t}h_L^t$. Thus for $a,b,s\in[0,1]$,
\(
 M_s(K_a,K_b)\subset K_{(1-s)a+sb}.
\)
Theorem~\ref{thm:intrinsic} implies that $\varphi$ is concave.
Apply \Cref{lem:variation} to $f_t=h_K^{1-t}h_L^t$, 
\begin{equation}\label{eq:derivative}
 \varphi'_+(0)=\frac1{V(K)}\int_{\Sph^{n-1}}
 h_K\log\frac{h_L}{h_K}\dd S_K
 =\frac n{V(K)}\int_{\Sph^{n-1}}\log\frac{h_L}{h_K}\dd V_K.
\end{equation}
Concavity gives $\varphi'_+(0)\ge\varphi(1)-\varphi(0)$, proving
\eqref{eq:logmink}. If equality holds, $\varphi'(s)$ is constant on $[0,1]$, making $\varphi$ affine. Equality in
Theorem~\ref{thm:intrinsic} gives \eqref{eq:split}. Conversely,
\eqref{eq:exactsum} makes $\varphi$ affine, and \eqref{eq:derivative} gives equality in \eqref{eq:logmink}.
\end{proof}

\begin{proof}[Proof of Corollary~\ref{cor:cone}]
If $V_K=V_L$, their total masses give $V(K)=V(L)$.
Applying \eqref{eq:logmink} to $(K,L)$ and $(L,K)$ gives
\[
 \int_{\Sph^{n-1}}\log\frac{h_L}{h_K}\,dV_K\ge0,
 \qquad
 \int_{\Sph^{n-1}}\log\frac{h_K}{h_L}\,dV_K\ge0,
\]
 so both vanish.
The equality statement in Theorem~\ref{thm:logmink} therefore
yields the decomposition \eqref{eq:split}.
Writing $d_j=\dim F_j$, equality of total masses gives
\[
 \prod_{j=1}^m c_j^{d_j}
 =\frac{V(L)}{V(K)}
 =1.
\]

Conversely, suppose that \eqref{eq:split} and
\eqref{eq:productone} hold. First assume that the subspaces
$F_1,\ldots,F_m$ are mutually orthogonal, so that $K$ is
identified with the Cartesian product of the $K_j$.
Up to a set of $\mathcal H^{n-1}$-measure zero, every point
of $\partial K$ has exactly one component on the boundary
($\partial K_i$) and other components in the corresponding
relative interiors ($\operatorname{relint}K_j(j\ne i)$). Indeed, the boundary points on which
two or more components lie on their respective boundaries
have zero $\mathcal H^{n-1}$-measure.

For every Borel set
$\omega\subset F_i\cap\Sph^{n-1}$, Fubini's theorem gives
\[
 \begin{aligned}
 V_K(\omega)
 &=
 \frac1n
 \left(\prod_{j\ne i}V_{d_j}(K_j)\right)
 \int_{\nu_{K_i}^{-1}(\omega)}
 \langle x_i,\nu_{K_i}(x_i)\rangle\,
 d\mathcal H^{d_i-1}(x_i)\\
 &=
 \frac{d_i}{n}
 \left(\prod_{j\ne i}V_{d_j}(K_j)\right)
 V_{K_i}(\omega).
 \end{aligned}
\]
Here $V_{K_i}$ is the cone volume measure of $K_i$ in $F_i$.
Thus
\begin{equation}\label{eq:productmeasure}
 V_K\big|_{F_i\cap\Sph^{n-1}}
 =
 \frac{d_i}{n}
 \left(\prod_{j\ne i}V_{d_j}(K_j)\right)V_{K_i}.
\end{equation}
This computation also applies when $d_i=1$, with
$\mathcal H^0$ interpreted as counting measure on the two
endpoints. Moreover, $V_K$ is concentrated on
$\bigcup_i(F_i\cap\Sph^{n-1})$.

Under the dilation $K_j\mapsto c_jK_j$, we have
\[
 V_{d_j}(c_jK_j)=c_j^{d_j}V_{d_j}(K_j),
 \qquad
 V_{c_jK_j}=c_j^{d_j}V_{K_j}.
\]
Applying \eqref{eq:productmeasure} to $L$ therefore yields
\[
 V_L\big|_{F_i\cap\Sph^{n-1}}
 =
 \left(\prod_{j=1}^m c_j^{d_j}\right)
 V_K\big|_{F_i\cap\Sph^{n-1}}
 =
 V_K\big|_{F_i\cap\Sph^{n-1}}.
\]
leading to $V_L=V_K$.

To treat arbitrary subspace decomposition, let $\Phi$ be
an invertible linear map and define
\[
 N_\Phi(u)=\frac{\Phi^{-T}u}{|\Phi^{-T}u|},
 \qquad u\in\Sph^{n-1}.
\]
If $u$ is the outer unit normal at a regular point
$x\in\partial K$, then $N_\Phi(u)$ is the outer unit normal
at $\Phi x\in\partial(\Phi K)$. Hence
\begin{equation}\label{eq:conecovariance}
 V_{\Phi K}=|\det\Phi|\,(N_\Phi)_\#V_K.
\end{equation}
Choose $\Phi$ so that the subspaces $\Phi F_i$ are mutually
orthogonal. The preceding argument gives
$V_{\Phi K}=V_{\Phi L}$. Since $N_\Phi$ is a homeomorphism
of the sphere, \eqref{eq:conecovariance} then implies
$V_K=V_L$.

Finally, a direct sum with at least two convex body contains a nondegenerate line segment in its boundary.
 Thus, if either $K$ or $L$ is strictly
convex, $V_K=V_L$ implies $K=L$.
If the cone volume measure has full support, a
decomposition with two or more subspaces is also impossible.
Indeed, in orthogonal coordinates,
\eqref{eq:productmeasure} concentrates the measure on the
union of \(F_i\cap\Sph^{n-1}\),
which is a proper closed subset of the sphere when $m\ge2$.
Since $N_\Phi$ is a homeomorphism,
a measure supported on proper closed subset
cannot become a full-support measure under an invertible linear change of coordinates.
Thus $V_K=V_L$ implies $K=L$.

\end{proof}

\begin{proof}[Proof of \Cref{thm:revolution-rigidity}]
By Corollaries~\ref{cor:revolution} and~\ref{cor:cone},
the general equality and cone volume classifications apply.
We first show that an origin-symmetric body of revolution
admits a nontrivial direct-sum decomposition only if it is
a circular cylinder, and that it must be
the axial segment and the equatorial ball.

Suppose that
\( K=K_1\oplus\cdots\oplus K_m, 
 \mathbb R^n=F_1\oplus\cdots\oplus F_m, \)with $m\ge2,$ 
where $F_i=\operatorname{span}K_i$ and $\dim F_i\ge1$.
Define
\[
 G_i=\left(\sum_{j\ne i}F_j\right)^\perp.
\]
These are complementary proper subspaces satisfying
$\dim G_i=\dim F_i$.

As in the proof of
Corollary~\ref{cor:cone}, for
$\mathcal H^{n-1}$-almost every $x\in\partial K$, the
decomposition $x=\sum_jx_j$ has exactly one component
$x_i\in\partial K_i$, while all other components belong
to the corresponding relative interiors.
Any outer normal at such a point is perpendicular to
$F_j$ for every $j\ne i$, thus this normal belongs to $G_i$ and
\begin{equation}\label{eq:product-normal-support}
 \operatorname{supp}S_K
 \subset
 \bigcup_{i=1}^m
 \bigl(G_i\cap\mathbb S^{n-1}\bigr).
\end{equation}

Let $\mathbb Ra$ be an axis of revolution of $K$, with
$|a|=1$. Suppose that $u=ta+v\in\operatorname{supp}S_K$
with $t\ne0$ and $v\in a^\perp\setminus\{0\}$.
Rotational invariance of $K$ implies that the entire orbit
\( \mathcal O
 =
 \{ta+w:w\in a^\perp,\ |w|=|v|\}\)
is contained in $\operatorname{supp}S_K$.
This orbit is an $(n-2)$-dimensional sphere.

For each $i$, choose $0\ne b_i\in G_i^\perp$ and write
$b_i=\alpha_i a+\beta_i$, where $\beta_i\in a^\perp$.
If $ta+w\in\mathcal O\cap G_i$, then
\(t\alpha_i+\langle w,\beta_i\rangle=0.\)
When $\beta_i=0$, this equation has no solution.
When $\beta_i\ne0$, it restricts $w$ to a hyperplane section
of the sphere $\{w\in a^\perp:|w|=|v|\}$, which has zero
$\mathcal H^{n-2}$-measure.
Thus every $\mathcal O\cap G_i$ has zero
$\mathcal H^{n-2}$-measure, whereas $\mathcal O$ has positive
$\mathcal H^{n-2}$-measure, contradicting \eqref{eq:product-normal-support}.
Consequently,
\(
 \operatorname{supp}S_K
 \subset
 \bigl(\mathbb Ra\cup a^\perp\bigr)
 \cap\mathbb S^{n-1}.
\)Both axial and equatorial directions must occur in
$\operatorname{supp}S_K$, otherwise the intersection \eqref{eq:supp}
would be unbounded in the axial or equatorial directions,
respectively.
Origin symmetry supplies both axial directions $\pm a$,
and rotational invariance supplies the entire equatorial
sphere once one equatorial direction occurs. Hence
\[
 \operatorname{supp}S_K
 =
 \{\pm a\}\cup
 \bigl(a^\perp\cap\mathbb S^{n-1}\bigr).
\]

Set $h=h_K(a)=h_K(-a)>0$.
Rotational invariance also implies that $h_K$ has a
constant value $r>0$ on $a^\perp\cap\mathbb S^{n-1}$.
Therefore,
\[
 \begin{aligned}
 K
 &=
 \{sa+y:|s|\le h,\ y\in a^\perp,\ |y|\le r\}\\
 &=[-ha,ha]+rB_{a^\perp}.
 \end{aligned}
\]
Thus $K$ is a circular cylinder.
Since $a^\perp\cap S^{n-1}\subset\bigcup_iG_i$, some $G_i$ contains $a^\perp$;
 otherwise all intersections have zero $(n-2)$-dimensional spherical measure.
  After relabeling, $G_1=a^\perp$. Since the positive dimensions of the $G_i$ sum to $n$, we have $m=2$ and $\dim G_2=1$.
   The inclusion $a\in G_1\cup G_2$ gives $G_2=\R a$. Consequently $F_1=a^\perp$, $F_2=\R a$
and \(K_1=rB_{a^\perp}, K_2=[-ha,ha].\)

The equality classification now gives two possibilities.
If there is only one subspace, then $L=cK$ for some $c>0$.
Otherwise, the preceding argument shows that $K$ is a
circular cylinder and that $L$ is obtained by independently
dilating its axial and equatorial components. Thus
\[
 K=[-ha,ha]+rB_{a^\perp},
 \qquad
 L=[-\widetilde h a,\widetilde h a]
       +\widetilde rB_{a^\perp}
\]
for some $h,r,\widetilde h,\widetilde r>0$.
In particular, $K$ and $L$ are coaxial circular cylinders.
Conversely, dilates and coaxial circular cylinders satisfy
the direct-sum equality condition, so equality holds in
the logarithmic Brunn--Minkowski inequality.

Finally, suppose that $V_K=V_L$.
For dilates case, $K=L$.
In the coaxial cylinders case, the axial and equatorial
dilation ratios are $\widetilde h/h$ and
$\widetilde r/r$, with dimensions $1$ and $n-1$,
respectively. Equality of the cone volume measure implies
\(\frac{\widetilde h}{h}
 \left(\frac{\widetilde r}{r}\right)^{n-1}=1\), equivalent to
\( hr^{n-1}=\widetilde h\,\widetilde r^{\,n-1}.\)
\end{proof}

\begin{proof}[Proof of Theorem~\ref{thm:simplex}]
We first show that $K$ admits no nontrivial direct-sum decomposition.
Suppose, to the contrary, that
\(
 K=K_1\oplus\cdots\oplus K_m,
 \mathbb R^n=F_1\oplus\cdots\oplus F_m,
\)
with $m\ge2$,
where each $K_i\subset F_i$ has nonempty relative interior
and $\dim F_i\ge1$.
Let $\pi_i$ be the linear projection onto $F_i$ along
$\bigoplus_{j\ne i}F_j$.
Then
\(
 K_i=\pi_iK,
\)
so every $K_i$ is a polytope.

Choose one-dimensional faces $E_1$ of $K_1$ and $E_2$
of $K_2$, taking $E_i=K_i$ if $\dim F_i=1$.
For each $j\ge3$, choose a vertex $p_j$ of $K_j$.
These choices determine a face of $K$,
\[
 Q=E_1+E_2+\sum_{j=3}^m p_j.
\]
Indeed, choose linear functionals
$\ell_i:F_i\to\mathbb R$ exposing $E_i$ for $i=1,2$
and exposing $\{p_i\}$ for $i\ge3$.
When $E_i=K_i$, take $\ell_i=0$.
The functional $\ell$ on $\mathbb R^n$ defined by
\(
 \ell(x_1+\cdots+x_m)
 =\sum_{i=1}^m\ell_i(x_i)
\) where $x_i\in F_i$,
then satisfies
\[
 \left\{x\in K:\ell(x)=\max_{y\in K}\ell(y)\right\}
 =
 E_1+E_2+\sum_{j=3}^m p_j
 =Q.
\]

The directions  of $E_1$ and $E_2$ are linearly independent, since
$F_1\cap F_2=\{0\}$. Thus $Q$ is a nondegenerate
parallelogram of dimension two.
Because $n\ge3$, it is a proper face of $K$.
Every proper face of a simplicial polytope is a simplex.
In particular, every two-dimensional proper face of $K$
is a triangle. This contradicts the fact that $Q$ is
a parallelogram, and proves the claim.

Suppose that equality holds in the log--Brunn--Minkowski inequality for some $0<\lambda<1$.
The equality in
Theorem~\ref{thm:intrinsic} gives a decomposition
\[    
 K=K_1\oplus\cdots\oplus K_m,
 \qquad
 L=c_1K_1\oplus\cdots\oplus c_mK_m,
 \qquad c_i>0.
\]
 By the preceding claim, $m=1$, and hence
$L=c_1K$.
Conversely, dilates attain equality by the proof in Theorem~\ref{thm:main}.
Finally, suppose that $V_K=V_L$.
Corollary~\ref{cor:cone} gives \(c_1^n=1\), therefore $K=L$.
\end{proof}

\bibliographystyle{amsplain}
\bibliography{references}

\end{document}